\documentclass[11pt,reqno]{amsart}

\usepackage{mathrsfs}
\usepackage{amssymb}
\usepackage{amsthm}
\usepackage{amsmath,amsfonts,amssymb,esint}
\usepackage{graphics,color}
\usepackage{enumerate}

\usepackage{marginnote}
\usepackage{xfrac}
\usepackage{mathtools}
\usepackage{hyperref}

\usepackage{comment}

\newtheorem{theorem}{Theorem}[section]
\newtheorem{lemma}[theorem]{Lemma}
\newtheorem{proposition}[theorem]{Proposition}
\newtheorem{corollary}[theorem]{Corollary}
\newtheorem{definition}[theorem]{Definition\rm}
\newtheorem{remark}{Remark}

\newcommand{\T}{\ensuremath{\mathbb{T}}}
\newcommand*{\R}{\ensuremath{\mathbb{R}}}

\newcommand*{\N}{\ensuremath{\mathbb{N}}}
\newcommand*{\Z}{\ensuremath{\mathbb{Z}}}
\newcommand*{\C}{\ensuremath{\mathbb{C}}}

\renewcommand*{\div}{\ensuremath{\mathrm{div\,}}}

\newcommand*{\Id}{\ensuremath{\mathrm{Id}}}

\newcommand{\eps}{\varepsilon}

\renewcommand*{\P}{\ensuremath{\mathcal{P}}}

\newcommand*{\RR}{\ensuremath{\mathcal{R}}}
\newcommand{\norm}[1]{\left\|#1\right\|}
\newcommand{\abs}[1]{\left|#1\right|}
\DeclarePairedDelimiter{\floor}{\lfloor}{\rfloor}
\newcommand{\ip}[1]{\left\langle#1\right\rangle}

\DeclareMathOperator{\support}{support}

\begin{document}

\begin{abstract}
In recent works by Isett \cite{Isett}, and later by Buckmaster, De Lellis, Isett and Sz\'ekelyhidi Jr.\ \cite{BDIS}, iterative schemes were presented for constructing solutions belonging to the H\"older class $C^{\sfrac15-\eps}$ of the 3D incompressible Euler equations which do not conserve the total kinetic energy.  The cited work is partially motivated by a conjecture of Lars Onsager in 1949 relating to the existence of $C^{\sfrac13-\eps}$ solutions to the Euler equations which dissipate energy.  In this note we show how the later scheme can be adapted in order to prove the existence of non-trivial H\"older continuous solutions which for almost every time belong to the critical Onsager H\"older regularity $C^{\sfrac13-\eps}$ and have compact temporal support.  
\end{abstract}

\title[Onsager's conjecture almost everywhere in time]
{Onsager's conjecture almost everywhere in time}

\author{Tristan Buckmaster}
\address{Institut f\"ur Mathematik, Universit\"at Leipzig, D-04103 Leipzig}
\email{tristan.buckmaster@math.uni-leipzig.de}

\maketitle

\section{Introduction}
\nocite{bds2}

In what follows $\T^3$ denotes the $3$-dimensional torus, i.e. $\T^3 = {\mathbb S}^1\times {\mathbb S}^1 \times
{\mathbb S}^1$.  Formally, we say $(v,p)$ solves the \emph{incompressible Euler equations} if
\begin{equation}\label{eulereq}
\left\{\begin{array}{l}
\partial_t v+\div v\otimes v +\nabla p =0\\
\div v = 0
\end{array}\right..
\end{equation}

Suppose $v$ is such a solution, then we define its \emph{kinetic energy}, as
\[E(t):=\frac{1}{2}\int_{\T^3} \abs{v(x,t)}^2~dx.\]
A simple calculation applying integration by parts yields that for any classical solution of \eqref{eulereq} the kinetic energy is in fact conserved in time.  This formal calculation does not however hold for distributional solutions to Euler (cf. \cite{Scheffer93,Shnirelmandecrease,DS1,DS2,Wiedemann,DSsurvey}).  

In fact in the context of 3-dimensional turbulence, flows  \emph{dissipating} energy in time have long been considered. A key postulate of Kolmogorov's K41 theory \cite{Kolmogorov} is that for homogeneous, isotropic turbulence, the dissipation rate is non-vanishing in the inviscid limit.  In particular, defining the \emph{structure functions} for homogeneous, isotropic turbulence
\[S_p(\ell):=\ip{\left[(v(x+\hat \ell)-v(x))\cdot\frac {\hat \ell}{\ell}\right]^p},\]
where $\hat \ell$ denotes a spatial vector of length $\ell$, Kolmogorov's famous four-fifths law can be stated as
\begin{equation}\label{e:four-fifths}
S_3(\ell)=-\frac45\eps_d \ell,
\end{equation}
where here $\eps_d$ denotes the mean energy dissipation per unit mass.  More generally, Kolmogorov's scaling laws can be stated as
\begin{equation}\label{e:scaling_laws}
S_p(\ell)= C_p \eps_d^{\sfrac p3} \ell^{\sfrac p3},
\end{equation}
for any positive integer $p$.

A well known consequence of the above scaling laws is the Kolmogorov spectrum, which postulates a scaling relation on the `energy spectrum' of a turbulent flow (cf.\ \cite{FrischBook,EyinkSreenivasan}).  It was this observation that provided motivation for Onsager to conjecture in his famous note \cite{Onsager} on statistical hydrodynamics,  the following dichotomy:
\begin{enumerate}
\item[(a)] Any weak solution $v$ belonging to the H\"older space $C^\theta$ for $\theta>\frac{1}{3}$ conserves the energy.
\item[(b)] For any $\theta<\frac{1}{3}$ there exist weak solutions $v\in C^\theta$ which do not conserve the energy.
\end{enumerate}

Part (a) of this conjecture has since been resolved: it was first considered by Eyink in \cite{Eyink} following Onsager's original calculations and later proven by Constantin, E and Titi in \cite{ConstantinETiti}.  Subsequently, this later result was strengthened by showing that under weakened assumptions on $v$ (in terms of Besov spaces) kinetic energy is conserved  \cite{RobertDuchon,CCFS2007}.  

Part (b) remains an open conjecture and is the subject of this note.  The first constructions of non-conservative H\"older-continuous ($C^{\sfrac{1}{10}-\eps}$) weak solutions appeared in work of De Lellis and Sz\'ekelyhidi Jr.\ \cite{DS4}, which itself was based on their earlier seminal work \cite{DS3} where continuous weak solutions were constructed.  Furthermore, it was shown in the mentioned work that such solutions can be constructed obeying any prescribed smooth non-vanishing energy profile.  In recent work \cite{Isett}, P. Isett introduced a number of new ideas in order to construct non-trivial $1/5-\eps$  H\"older-continuous weak solutions with compact temporal support.  This construction was later improved by Buckmaster,  De Lellis and Sz\'ekelyhidi Jr.\ \cite{BDIS}, following more closely the earlier work \cite{DS3,DS4}, in order construct $1/5-\eps$  H\"older-continuous weak solution obeying a given energy profile. 

In this note we give a proof of the following theorem.  

\begin{theorem}\label{t:main}
There exists is a non-trivial continuous vector field $v\in C^{\sfrac{1}{5}-\eps} ( \T^3 \times (-1, 1), \R^3)$ with compact support in time and a continuous
scalar field $p\in C^{\sfrac{2}{5}-2\eps} (\T^3\times (-1,1))$ with the following properties:
\begin{itemize}
\item[(i)]  The pair $(v,p)$ solves the incompressible
Euler equations \eqref{eulereq} in the sense of distributions.
\item[(ii)] There exists a set $\Omega\subset (-1,1)$ of Hausdorff dimension strictly less than $1$ such that if $t\notin\Omega$ then $v(\cdot,t)$ is H\"older $C^{1/3-\eps}$ continuous and $p$ is H\"older $C^{2/3-2\eps}$ continuous.\footnote{More precisely, the Hausdorff dimension $d$ is such that $1-d>C\eps^2$ for some positive constant $C$.}
\end{itemize}
\end{theorem}

\medskip

\noindent{\bf Relation to intermittency.} The theory of intermittency is born of an effort to explain the experimental and numerical evidence (e.g.\ \cite{benzi1993extended}) of measurable discrepancies  from the scaling laws \eqref{e:scaling_laws}  (cf.\  \cite{frisch1980fully}).  In this direction, Mandelbrot conjectured \cite{mandelbrot1976turbulence} that at the inviscid limit, turbulence concentrates (\emph{in space}) on a fractal set of Hausdorff dimension strictly less than 3.  

It is interesting to note that the solutions constructed in order to prove the above theorem have a fractal structure in \emph{time}: namely, the set of times for which $v$ is \emph{not} H\"older $C^{1/3-\eps}$ continuous is  contained in a Cantor-like set with Hausdorff dimension strictly less than 1.  Since the phenomena observed does not relate to the structure functions from which intermittency was originally postulated it is clearly far-fetched to label such a phenomena as intermittency.  Nevertheless, it is the opinion of the author that the parallels to the notion of intermittency  remain of interest.

\subsection{Euler-Reynolds system and the convex integration scheme}\label{s:setup}
In order to prove Theorem \ref{t:main} we construct an iteration scheme in the style of \cite{BDIS}, which is itself based on the schemes presented in \cite{DS3,DS4}. At each step $q\in \N$ we construct a triple $(v_q, p_q, \mathring{R}_q)$ solving the Euler-Reynolds system (see Definition 2.1 in \cite{DS3}):
\begin{equation}\label{e:euler_reynolds}
\left\{\begin{array}{l}
\partial_t v_q + \div (v_q\otimes v_q) + \nabla p_q =\div\mathring{R}_q\\ \\
\div v_q = 0\, .
\end{array}\right.
\end{equation} 

The initial triple $(v_0,p_0,\mathring{R}_0)$  will be non-trivial with compact support in time; all triples thereafter will be defined inductively as perturbations of the proceeding triples.  The perturbation 
$$
w_q:=v_{q}-v_{q-1},
$$
will be composed of weakly interacting perturbed Beltrami flows (see Section \ref{s:prelim}) oscillating at \emph{frequency} $\lambda_q$, defined in such a way to correct for the previous Reynolds error $\mathring{R}_{q-1}$.

In order to ensure convergence of the sequence $v_q$ to a continuous weak $C^{\sfrac15-\eps}$ solution of Euler, we will require the following estimates to be satisfied
\begin{align}
\|w_{q}\|_0 + \frac{1}{\lambda_q}\|\partial_t w_{q}\|_0 + \frac{1}{\lambda_q}\|w_{q}\|_1 &\leq \lambda_q^{-\sfrac15+\eps_0} \label{e:v_iter}\\
\norm{p_q-p_{q-1}}_0+ \frac{1}{\lambda_q}\norm{\partial_t(p_q-p_{q-1})}_0+ \frac{1}{\lambda_q^2}\norm{p_q-p_{q-1}}_2 &\leq \lambda_q^{-\sfrac25+2\eps_0} \label{e:p_iter}\\
\norm{\mathring R_q}_0+\frac{1}{\lambda_q}\norm{\mathring R_q}_1&\leq \lambda_{q+1}^{-\sfrac25+2\eps_0} \label{e:R_iter}
\end{align}
for some $\eps_0>0$ strictly smaller than $\eps$.  Here and throughout the article, $\norm{\cdot}_\beta$ for $\beta=m+\kappa$, $\beta\in \N$ and $\kappa\in [0,1)$ will denote the usual \emph{spatial} H\"older $C^{m,\kappa}$ norm.  As a minor point of deviation from \cite{BDIS}, we keep track of second order spatial derivative estimates of $p_q-p_{q-1}$, whereas in \cite{BDIS} first order estimates -- which in the present work are implicit by interpolation -- were sufficient.  These second order estimates will be used in order to obtain slightly improved bounds on the Reynolds stress (see Section \ref{s:reynolds}).  

It is perhaps worth noting that aside from the second order estimate on $p_q-p_{q-1}$, up to a constant multiple, the above estimates are consistent with the estimates given in \cite{BDIS}.\footnote{In \cite{BDIS} the estimates corresponding to \eqref{e:v_iter}-\eqref{e:R_iter} are written in terms of a sequence of parameters $\delta_q$ which in the context of the present paper are defined to be $\delta_q:=\lambda_q^{-\sfrac25+2\eps_0}$ (cf.\ Section \ref{s:ordering} and Section \ref{s:conclusion}).} 

In order to ensure that our sequence convergences to a non-trivial solution, we will impose the addition requirement that
\begin{equation}\label{e:nontrivial_req}
\sum_{q=1}^{\infty} \norm{w_q}_0 < \frac{1}{2}\norm{v_0}_0,
\end{equation}
for times $t\in [-\sfrac18,\sfrac18]$.

The principle new idea of this work is that in addition to the estimates given above, we will keep track of sharper, time localized estimates.  As a consequence of these sharper estimates, it can be shown that for any given time $t\in(-1,1)$ outside a prescribed set $\Omega$ of Hausdorff dimension strictly less than $1$,  there exists a $N=N(t)$ such that  
\begin{align}
\|w_{q}\|_0 + \frac{1}{\lambda_q}\|\partial_t w_{q}\|_0 + \frac{1}{\lambda_q}\|w_{q}\|_1 &\leq \lambda_q^{-\sfrac13+\eps_0}\label{e:sharp1} \\
\norm{p_q-p_{q-1}}_0+ \frac{1}{\lambda_q}\norm{\partial_t(p_q-p_{q-1})}_0+ \frac{1}{\lambda_q^2}\norm{p_q-p_{q-1}}_2 &\leq \lambda_q^{-\sfrac23+2\eps_0} \\
\norm{\mathring R_q}_0+\frac{1}{\lambda_q}\norm{\mathring R_q}_1&\leq \lambda_{q+1}^{-\sfrac23+2\eps_0}, \label{e:sharp3}\footnotemark
\end{align}
for\footnotetext{Here and throughout the paper we suppress the dependence on the time variable $t$.} every $q\geq N$. 

\subsection{The main iteration proposition and the proof of Theorem \ref{t:main}}

\begin{proposition}\label{p:iterate}
For every small $\eps_0>0$, there exists an $\alpha>1$, $d<1$ and a sequence of parameters $\lambda_0,\lambda_1,\dots$ satisfying $\sfrac12 \lambda_0^{\alpha^q}<\lambda_q<2\lambda_0^{\alpha^q}$ such that the following holds.  A sequence of triples $(v_q, p_q, \mathring{R}_q)$ can be constructed with temporal support confined to $[-\sfrac12,\sfrac12]$ solving \eqref{e:euler_reynolds} and satisfying the estimates (\ref{e:v_iter}-\ref{e:nontrivial_req}).  Moreover, for any $\delta>0$, there exists an integer $M$ such that if $\Xi^M$ denotes the set of times $t$ such that there exists a $q\geq M$ satisfying either
\begin{equation}
\begin{split}\label{e:onsager_est}
\|w_{q}\|_0 +  \frac{1}{\lambda_q}\|w_{q}\|_1 &> \lambda_q^{-\sfrac13+\eps_0},\ \text{ or}\\
\norm{p_q-p_{q-1}}_0+  \frac{1}{\lambda_q}\norm{p_q-p_{q-1}}_1 &> \lambda_q^{-\sfrac23+2\eps_0},
\end{split}
\end{equation}
then there exists a cover of $\Xi^M$ consisting of a sequence of balls of radius $r_i$ such that
\begin{equation}\label{e:real_Haus}
\sum r_i^d < \delta.
\end{equation}
\end{proposition}

\begin{proof}[Proof of Theorem \ref{t:main}]

Fix $\eps_0=\sfrac{\eps}{2}$ and let $(v_q, p_q,\mathring{R}_q)$  be a sequence as in Proposition \ref{p:iterate}.
It follows then easily that $(v_q, p_q)$ converge uniformly to a pair of continuous functions $(v,p)$ satisfying
\eqref{eulereq}, having compact temporal support. Moreover, by interpolating the inequalities \eqref{e:v_iter} and \eqref{e:p_iter} we obtain that $v_q$ converges in $C^{\sfrac15-\eps}$ and $p_q$ in $C^{\sfrac25-2\eps}$.

In order to prove (ii) we first fix $\delta>0$ and let $M$ and $\Xi^M$ be as in Proposition \ref{p:iterate}. Hence by assumption if $t\notin \Xi^M$
\begin{equation}
\begin{split}
\|w_{q}\|_0 +  \frac{1}{\lambda_q}\|w_{q}\|_1 &\leq \lambda_q^{-\sfrac13+\eps_0}\\
\norm{p_q-p_{q-1}}_0+  \frac{1}{\lambda_q}\norm{p_q-p_{q-1}}_1 &\leq \lambda_q^{-\sfrac23+2\eps_0},
\end{split}
\end{equation}
for all $q\geq M$.  Thus interpolating the inequalities above we obtain that $v-v_M$ is bounded in $C^{\sfrac13-\eps}$ and $p-p_M$ in $C^{\sfrac23-2\eps}$.  By \eqref{e:v_iter} and \eqref{e:p_iter}, the pair $(v_M,p_M)$ are $C^1$ bounded and thus it follows that $v$ and $p$ are bounded in $C^{\sfrac13-\eps}$ and $C^{\sfrac23-2\eps}$ respectively.  Letting $\delta$ tend to zero we obtain our claim.
\end{proof}

\subsection{Plan of the paper}

After recalling in Section \ref{s:prelim} some preliminary notation from the paper \cite{DS3}, in Section \ref{s:const_triple} we give the precise definition of the sequence of triples $(v_{q}, p_{q}, \mathring{R}_{q})$. In Section \ref{s:ordering} we list a number of inequalities that we will require on the various parameters of our scheme. The Sections \ref{s:perturbation_estimates} and \ref{s:reynolds} will focus on estimating, respectively, $w_{q+1} =v_{q+1}-v_q$, and $\mathring{R}_{q+1}$.
These estimates are then collected in Section \ref{s:conclusion} where Proposition \ref{p:iterate} will be finally proved. Throughout the entire article we will rely heavily on the arguments of \cite{BDIS} -- in some sense the scheme presented here is a simple variant of that given in \cite{BDIS} -- as such the present paper is intentionally structured in a similar manner to \cite{BDIS} in order to aide comparison.

\subsection{Acknowledgments}
I wish to thank Camillo De Lellis and L\'aszl\'o Sz\'ekelyhidi Jr.\ for the enlightening discussions I had with them both.  I would also like to thank Antoine Choffrut, Camillo De Lellis, Charles Doering and L\'aszl\'o Sz\'ekelyhidi Jr.\ for their helpful comments and corrections regarding the manuscript. In addition, I would like to express my gratitude to the anonymous referee for his/her detailed comments and corrections.

This work is supported as part of the ERC Grant Agreement No. 277993.

\section{Preliminaries}\label{s:prelim}

Throughout this paper we denote the $3\times 3$ identity matrix by $\Id$.   In this section we state a number of results found in \cite{DS3} which are fundamental to the present scheme as well its predecessors \cite{DS3,DS4,BDIS}.

\subsection{Geometric preliminaries}

The following two results will form the cornerstone in which to construct the highly oscillating flows required by our scheme. 

\begin{proposition}[Beltrami flows]\label{p:Beltrami}
Let $\bar\lambda\geq 1$ and let $A_k\in\R^3$ be such that 
$$
A_k\cdot k=0,\,|A_k|=\tfrac{1}{\sqrt{2}},\,A_{-k}=A_k
$$
for $k\in\Z^3$ with $|k|=\bar\lambda$.
Furthermore, let 
$$
B_k=A_k+i\frac{k}{|k|}\times A_k\in\C^3.
$$
For any choice of $a_k\in\C$ with $\overline{a_k} = a_{-k}$ the vector field
\begin{equation}\label{e:Beltrami}
W(\xi)=\sum_{|k|=\bar\lambda}a_kB_ke^{ik\cdot \xi}
\end{equation}
is real-valued, divergence-free and satisfies
\begin{equation}\label{e:Bequation}
\div (W\otimes W)=\nabla\frac{|W|^2}{2}.
\end{equation}
Furthermore
\begin{equation}\label{e:av_of_Bel}
\langle W\otimes W\rangle= \fint_{\T^3} W\otimes W\,d\xi = \frac{1}{2} \sum_{|k|=\bar\lambda} |a_k|^2 \left( \Id - \frac{k}{|k|}\otimes\frac{k}{|k|}\right)\, .  
\end{equation}
\end{proposition}

\begin{lemma}[Geometric Lemma]\label{l:split}
For every $N\in\N$ we can choose $r_0>0$ and $\bar{\lambda} > 1$ with the following property.
There exist pairwise disjoint subsets 
$$
\Lambda_j\subset\{k\in \Z^3:\,|k|=\bar{\lambda}\} \qquad j\in \{1, \ldots, N\}
$$
and smooth positive functions 
\[
\gamma^{(j)}_k\in C^{\infty}\left(B_{r_0} (\Id)\right) \qquad j\in \{1,\dots, N\}, ~k\in\Lambda_j,~\footnotemark
\]
such \footnotetext{Here $B_{r_0} (\Id)$ denotes the ball around $\Id$ of radius $r_0$ under the usual matrix operator norm $|A|:=\max_{|v|=1}|Av|$.}that
\begin{itemize}
\item[(a)] $k\in \Lambda_j$ implies $-k\in \Lambda_j$ and $\gamma^{(j)}_k = \gamma^{(j)}_{-k}$;
\item[(b)] For each $R\in B_{r_0} (\Id)$ we have the identity
\begin{equation}\label{e:split}
R = \frac{1}{2} \sum_{k\in\Lambda_j} \left(\gamma^{(j)}_k(R)\right)^2 \left(\Id - \frac{k}{|k|}\otimes \frac{k}{|k|}\right) 
\qquad \forall R\in B_{r_0}(\Id)\, .
\end{equation}
\end{itemize}
\end{lemma}

\subsection{The operator $\mathcal{R}$} The following operator will be used in order to deal the the Reynolds Stresses arising from our iteration scheme.

\begin{definition}\label{d:reyn_op}
Let $v\in C^\infty (\T^3, \R^3)$ be a smooth vector field. 
We then define $\RR v$ to be the matrix-valued periodic function
\begin{equation*}
\RR v:=\frac{1}{4}\left(\nabla\P u+(\nabla\P u)^T\right)+\frac{3}{4}\left(\nabla u+(\nabla u)^T\right)-\frac{1}{2}(\div u) \Id,
\end{equation*}
where $u\in C^{\infty}(\T^3,\R^3)$ is the solution of
\begin{equation*}
\Delta u=v-\fint_{\T^3}v\textrm{ in }\T^3
\end{equation*}
with $\fint_{\T^3} u=0$ and $\P$ is the Leray projection onto divergence-free fields with zero average.
\end{definition}

\begin{lemma}[$\RR=\textrm{div}^{-1}$]\label{l:reyn}
For any $v\in C^\infty (\T^3, \R^3)$ we have
\begin{itemize}
\item[(a)] $\RR v(x)$ is a symmetric trace-free matrix for each $x\in \T^3$;
\item[(b)] $\div \RR v = v-\fint_{\T^3}v$.
\end{itemize}
\end{lemma}

\subsection{Schauder and commutator estimates on $\mathcal R$}

We recall the following Schauder estimates (Proposition G.1 (ii), Appendix G of \cite{BDIS}) and commutator estimates (Proposition H.1 Appendix H of \cite{BDIS}).

\begin{proposition}\label{p:stat_phase}
Let $k\in\Z^3\setminus\{0\}$ be fixed. For a smooth vector field $a\in C^{\infty}(\T^3;\R^3)$ let 
$F(x):=a(x)e^{i\lambda k\cdot x}$. Then we have
\begin{equation}\label{e:R(F)}
\|\RR(F)\|_{\alpha}\leq \frac{C}{\lambda^{1-\alpha}}\|a\|_0+\frac{C}{\lambda^{m-\alpha}}[a]_m+\frac{C}{\lambda^m}[a]_{m+\alpha},
\end{equation}
for  $m=0,1,2,\dots$ and $\alpha\in (0,1)$, where $C=C(\alpha,m)$.
\end{proposition}

\begin{proposition}\label{p:commutator}
Let $k\in\Z^3\setminus\{0\}$ be fixed. For any smooth vector field $a\in C^\infty  (\T^3;\R^3)$ and any smooth function $b$, if we set $F(x):=a(x)e^{i\lambda k\cdot x}$, we then have
\begin{equation}
\|[b, \mathcal{R}] (F)\|_\alpha \leq  C\lambda^{\alpha-2}  \|a\|_0\|b\|_1
+ C \lambda^{\alpha-m} \left(\|a\|_{m-1+\alpha} \|b\|_{1+\alpha} + \|a\|_{\alpha} \|b\|_{m+\alpha}\right)\label{e:main_est_commutator}
\end{equation}
for  $m=0,1,2,\dots$ and $\alpha\in (0,1)$, where $C=C(\alpha,m)$.
\end{proposition}

\section{The construction of the triples $(v_q,p_q,\mathring R_q)$}\label{s:const_triple}

\subsection{The initial triple  $(v_0,p_0,\mathring R_0)$}
Let $\chi_0$ be a smooth non-negative function, compactly supported on the interval $[-\sfrac14,\sfrac14]$, bounded above by $1$ and identically equal to $1$ on $[-\sfrac18,\sfrac18]$.       We now set our initial velocity to be the divergence-free vector field 
\begin{equation*}
v_0 (t,x) := \frac12\lambda_0^{-\frac15+\eps_0}\chi_0(t)(\cos(\lambda_0 x_3),\sin(\lambda_0 x_3),0),
\end{equation*}
where here we use the notation $x=(x_1,x_2,x_3)$. The initial pressure $p_0$ is then defined to be identically zero.  Finally if we set
\[
\mathring R_0 =  \frac12\lambda_0^{-\frac65+\eps_0} \chi'_0(t)
 \begin{pmatrix}
  0 & 0 & \sin(\lambda_0 x_3) \\
  0 & 0 & -\cos(\lambda_0 x_3) \\
  \sin(\lambda_0 x_3) & -\cos(\lambda_0 x_3)  & 0
 \end{pmatrix},
\]
we obtain
\[
\partial_t v_0+\div (v_0\otimes v_0)+\nabla p_0= \div \mathring R_0.
\]
Hence the triple $(v_0,p_0,\mathcal R_0)$ is a solution to the Euler-Reynolds system \eqref{e:euler_reynolds}.   Furthermore, it follows immediately that
\[\norm{\mathring R_0}_0+\frac{1}{\lambda_0}\norm{\mathring R_0}_1\leq C \lambda_0^{-\sfrac65+\epsilon_0}.\]
Thus if $\lambda_0$ is sufficiently large we obtain (\ref{e:v_iter}-\ref{e:R_iter}) for $q=0$.  

\begin{remark}
The choice of initial triple $(v_0,p_0,\mathring R_0)$ is not of any great importance: any choice satisfying the conditions set out in Section \ref{s:setup} and is such that $\abs{v_0}\approx \lambda_0^{-\sfrac15+\eps_0}$ for times $t\in[-\sfrac18, \sfrac18]$ should suffice.
\end{remark}

%%%%%%%%%%%%%%%%
\subsection{The inductive step}\label{s:perturbations}
%%%%%%%%%%%%%%%%
The procedure of constructing $(v_{q+1}, p_{q+1}, \mathring{R}_{q+1})$ in terms of $(v_q, p_q, \mathring{R}_q)$ follows in the same spirit as that of the scheme outlined in \cite{BDIS} with a few minor modifications in order to satisfy the specific requirements of Proposition \ref{p:iterate}. 

We will assume that $\lambda_0$ is chosen large enough such that
\begin{equation}\label{e:summabilities}
\sum_{j< q} \lambda_j^{\sfrac23} \leq \lambda_q^{\sfrac23}\,   \quad \text {and } \sum_{j=1}^{\infty} \lambda_j^{-\sfrac15+\eps_0} \leq \frac{\lambda_0^{-\sfrac15+\eps_0} }{4}\leq \frac18.
\end{equation}
Notice as a direct consequence \eqref{e:nontrivial_req} follows from \eqref{e:v_iter} and the definition of $v_0$.

We fix a symmetric non-negative convolution kernel $\psi$ with support confined to $[-1,1]$.

With a slight abuse of notation, we will use $(v,p,\mathring{R})$ for $(v_q, p_q, \mathring{R}_q)$ and
$(v_1,p_1, \mathring{R}_1)$ for $(v_{q+1}, p_{q+1}, \mathring{R}_{q+1})$.

As was done in \cite{BDIS}, we discretize time into intervals of size $\mu^{-1}$ for some large parameter $\mu$ to be chosen later.

The choice of cut-off functions $\chi=\chi^{(q+1)}$ used in this article will differ slightly to that described in \cite{BDIS}. Specifically, we define $\chi$ to be a smooth function such that for a small parameter $\eps_1>0$  (to be chosen later) $\chi$ satisfies the following conditions:
\begin{itemize}
\item The support of $\chi$ is contained in $\left(-\frac12-\frac{\lambda_{q+1}^{-\eps_1}}4, \frac12+\frac{\lambda_{q+1}^{-\eps_1}}4\right)$.
\item In the range $\left(-\frac12+\frac{\lambda_{q+1}^{-\eps_1}}4, \frac12-\frac{\lambda_{q+1}^{-\eps_1}}4\right)$ we have $\chi\equiv 1$.
\item The sequence $\{\chi^2 (x-l)\}_{l\in\Z}$ forms a partition of unity of $\R$, i.e.\
\[
\sum_{l\in \Z} \chi^2 (x-l) = 1.
\]  
\item For $N\geq 0$ we have the estimates
\[\abs{\partial^N_x \chi}\leq C \lambda_{q+1}^{N\eps_1},\]
where the constant $C$ depends only on $N$ -- in particular it is independent of $q$.
\end{itemize}

In \cite{BDIS}, $\chi$ was simply chosen to be a $C_c^{\infty} (-\frac34, \frac34)$ function, independent of the iteration $q$, satisfying the third condition. 
Having defined $\chi$, we adopt the notation $\chi_l(t):=\chi(\mu t-l)$.  The fundamental difference to choice of $\chi$ in \cite{BDIS} is the extra factor $\lambda_{q+1}^{-\eps_1}$ appearing in the definition.    A consequence of this modification is that the Lebesgue measure of the set
\begin{equation*}
\bigcap_{q=1}^\infty \bigcup_{q'=q}^\infty \bigcup_l \support (\chi_{q', l}') 
\end{equation*}
is zero.  We will see this will provide us with a key ingredient in order to prove a.e.\ in time $C^{\sfrac13-\eps}$ convergence of the sequence $v_q$.

For each $l$ define the amplitude function
\[\rho_l=2r_0^{-1} \norm{\mathring{R}(\cdot,l\mu^{-1}) }_0.\]
The function $\rho_l$ will play a similar role to the $\rho_l$ found in \cite{BDIS}: the comparatively simpler definition above reflects the fact that we are only interested in correcting for the Reynolds error and are not attempting to construct a solution to Euler with a prescribed energy as was done in \cite{BDIS}.   In particular, up to a constant multiple, the amplitude function $\rho_{l}$ defined here provides a lower bound for the amplitude defined in \cite{BDIS}, and moreover is potentially significantly smaller.

Keeping in mind the new choices of $\rho_l$ and $\chi_l$, the construction of $(v_{1}, p_{1}, \mathring{R}_{1})$ proceeds in exactly the same manner as that described in \cite{BDIS}, with the minor exception that the mollification parameter $\ell$ will be chosen explicitly to be

\begin{equation}\label{e:ell_def}
\ell=\lambda_{q+1}^{-1+\eps_1}.
\end{equation}
 In particular assuming $\frac{\alpha-1}2>\eps_1$ and $\lambda_0$ is chosen sufficiently large, we have 
\begin{equation}\label{e:ell_lambda}
\frac1{\lambda_{q}}\leq \ell \leq \frac1{\lambda_{q+1}}.
\end{equation}
 For comparison, the choice of $\ell$ taken in \cite{BDIS} was $\ell := \delta_{q+1}^{-\sfrac{1}{8}}\delta_{q}^{\sfrac{1}{8}}\lambda_{q}^{-\sfrac{1}{4}}\lambda_{q+1}^{-\sfrac{3}{4}}$.   The parameter $\eps_1$ may be taken arbitrarily small, and consequently, the choice of $\ell$ taken here will be significantly smaller than that taken in \cite{BDIS}.

For completeness we recall the remaining steps required to construct the triple $(v_{1}, p_{1}, \mathring{R}_{1})$.

Having set
\[
R_l(x):= \rho_l \Id-  \mathring{R} (x,l\mu^{-1})\,
\]
and $v_{\ell}=v*\psi_{\ell}$, we define $R_{\ell,l}$ to be the unique solution to the following transport equation
\begin{equation*}
\left\{\begin{array}{l}
\partial_t R_{\ell,l} + v_{\ell}\cdot  \nabla R_{\ell,l} =0\\ \\
R_{\ell,l}(\frac l{\mu},\cdot)=R_l*\psi_{\ell}\, .
\end{array}\right.
\end{equation*}

For every integer $l\in [-\mu, \mu]$, we let $\Phi_l: \R^3\times (-1,1)\to \R^3$ be the solution of 
\begin{equation*}
\left\{\begin{array}{l}
\partial_t \Phi_l + v_{\ell}\cdot  \nabla \Phi_l =0\\ \\
\Phi_l (x,l \mu^{-1})=x.
\end{array}\right.
\end{equation*}

Applying Lemma \ref{l:split} with $N=2$, we denote by $\Lambda^e$ and $\Lambda^o$ the corresponding families of frequencies in $\Z^3$ and set $\Lambda := \Lambda^o$ + $\Lambda^e$. For each $k\in \Lambda$ and each $l\in \Z\cap[0,\mu]$ we then define
\begin{align*}
a_{kl}(x,t)&:=\sqrt{\rho_l}\gamma_k \left(\frac{R_{\ell,l}(x,t)}{\rho_l}\right),\\
w_{kl}(x,t)& := a_{kl}(x,t)\,B_ke^{i\lambda_{q+1}k\cdot \Phi_l(x,t)}.
\end{align*}
The perturbation $w=v_1-v$ is then defined as the sum of a ``principal part'' and a ``corrector".  The ``principal part'' being the map
\begin{align*}
w_o (x,t) := \sum_{\textrm{$l$ odd}, k\in \Lambda^o} \chi_l(t)w_{kl} (x,t) +
\sum_{\textrm{$l$ even}, k\in \Lambda^e} \chi_l(t)w_{kl} (x,t)\, .
\end{align*}
The ``corrector" $w_c$ is then defined in such a way that the sum $w= w_o+w_c$ is divergence free:
\begin{equation*}
w_c
=\sum_{kl}\chi_l\Bigl(\frac{i}{\lambda_{q+1}}\nabla a_{kl}-a_{kl}(D\Phi_{l}-\Id)k\Bigr)\times\frac{k\times B_k}{|k|^2}e^{i\lambda_{q+1}k\cdot\Phi_l}.
\end{equation*}

%%%%%%%%%%%%%%%%%%%%%%%%%%%%%%
The new pressure is defined as
\begin{equation*}
p_1=p-\frac{|w_o|^2}{2} - \frac{1}{3} |w_c|^2 - \frac{2}{3} \langle w_o, w_c\rangle - \frac{2}{3} \langle v-v_\ell, w\rangle \,.
\end{equation*}
and finally we set $\mathring{R}_1= R^0+R^1+R^2+R^3+R^4+R^5$, where
\begin{align}
R^0 &= \mathcal R \left(\partial_tw+v_\ell\cdot \nabla w+w\cdot\nabla v_\ell\right)\label{e:R^0_def}\\
R^1 &=\mathcal R \div \Big(w_o \otimes w_o- \sum_l \chi_l^2 R_{\ell, l} 
-\textstyle{\frac{|w_o|^2}{2}}\Id\Big)\label{e:R^1_def}\\
R^2 &=w_o\otimes w_c+w_c\otimes w_o+w_c\otimes w_c - \textstyle{\frac{|w_c|^2 + 2\langle w_o, w_c\rangle}{3}} {\rm Id}\label{e:R^2_def}\\
R^3 &= w\otimes (v - v_\ell) + (v-v_\ell)\otimes w
 - \textstyle{\frac{2 \langle (v-v_{\ell}), w\rangle}{3}} \Id\label{e:R^3_def}\\
R^4&=\mathring R- \mathring{R}* \psi_\ell \label{e:R^4_def}\\
R^5&=\sum_l \chi_l^2 (\mathring{R}_{\ell, l} + \mathring{R}*\psi_\ell)\label{e:R^5_def}\, .
\end{align}

\subsection{Compact support in time}

By construction it follows that if for each integer $j$ the triple $(v_j,p_j,\mathring R_j)$ is supported in the time interval $[T,T']$ then $(v_{j+1},p_{j+1},\mathring R_{j+1})$ is supported in the time interval $[T-\mu_{j+1}^{-1}, T'+\mu_{j+1}^{-1}]$.  Therefore since $(v_0,p_0,\mathring R_0)$ is supported in the time interval $[-\sfrac14,\sfrac14]$, it follows by induction that if we assume
\begin{equation}\label{e:mu_lower}
\mu_j\geq 2^{j+2}
\end{equation}
 then triple $(v_q,p_q,\mathring R_q)$ is supported in the time interval 
\[
\big[-\sfrac14-\sum_{j=1}^{q} \mu_j^{-1},\sfrac14+\sum_{j=1}^{q} \mu_j^{-1}\big]\subset[-\sfrac12, \sfrac12].
\]

\section{Ordering of parameters}\label{s:ordering}

In order to better aid comparison to arguments of \cite{BDIS}, we introduce a  sequence of \emph{strictly decreasing} parameters $\delta_q<1$.  In Section \ref{s:conclusion} we will provide an explicit definition of $\delta_q$, but for now we restrict ourselves to specifying a number of inequalities that $\delta_q$ will need to satisfy.  Analogously to  \cite{BDIS} we will assume the following estimates
\begin{align}
\frac{1}{\lambda_q}\norm{v_q}_1&\leq \delta_q^{\sfrac12}\label{e:delta_cond_first}\\
\frac{1}{\lambda_q}\norm{p_q}_1+\frac{1}{\lambda_q^2}\norm{p_q}_2&\leq \delta_q\label{e:delta_cond_second}\\
\norm{\mathring R_{q}}_0+\frac{1}{\lambda_{q}}\norm{\mathring R_{q}}_1&\leq \frac{1}C_0 \delta_{q+1}\label{e:iter_rey}\\
\norm{\partial_t+v\cdot\nabla \mathring R_q}_0&\leq \delta_{q+1}\delta_q^{\sfrac12}\lambda_q,\label{e:delta_cond_last}
\end{align}
where $C_0>1$ is a large number to be specified in the next section.

Furthermore, we will assume in addition that the following parameter inequalities are satisfied
\begin{equation}\label{e:conditions_lambdamu_2}
\begin{split}\sum_{j< q} \delta_j \lambda_j \leq  \delta_q \lambda_q,\qquad \frac{\delta_{q}^{\sfrac{1}{2}}\lambda_q\ell}{\delta_{q+1}^{\sfrac{1}{2}}}\leq1, \\
 \frac{\delta_q^{\sfrac12}\lambda_q}{\mu}\leq \lambda_{q+1}^{-\eps_1}\quad\mbox{and}\quad
\frac{1}{\lambda_{q+1}}\leq  \frac{\delta_{q+1}^{\sfrac12}}{\mu}.\end{split}
\end{equation}

The sequence $\delta_q$ will be applied in the context of proving $\sfrac15-\eps$ convergence of the velocities $v_q$; however note that unlike the case in \cite{BDIS}, the sequence does not appear explicitly in the definition of the triples $(v_q,p_q,\mathring R_q)$.   

In order to prove a.e. time $\sfrac13-\eps$
 convergence, we will require localized estimates (in time). To this aim, we fix a time $t_0\in(-1,1)$ and set $l_{q+1}$ to be the unique integer such that $\mu t_0\in [-\sfrac12+l_{q+1},\sfrac12+l_{q+1})$.  We now introduce a new sequence of strictly decreasing parameters $\delta_{q,t_0}\leq \delta_q$ such that for a given time $t$ satisfying $\abs{\mu t-l_{q+1}}\leq 1$ we have the following estimates
\begin{align}
\frac{1}{\lambda_q}\norm{v_q}_1&\leq \delta_{q,t_0}^{\sfrac12}\label{e:delta_cond_first2}\\
\frac{1}{\lambda_q}\norm{p_q}_1+\frac{1}{\lambda_q^2}\norm{p_q}_2&\leq \delta_{q,t_0}\label{e:delta_cond_second2}\\
\norm{\mathring R_{q}}_0+\frac{1}{\lambda_{q}}\norm{\mathring R_{q}}_1&\leq\frac{1}{C_0}  \delta_{q+1,t_0}\label{e:iter_rey2}\\
\norm{\partial_t+v\cdot\nabla \mathring R_q}_0&\leq \delta_{q+1,t_0}\delta_{q,t_0}^{\sfrac12}\lambda_q.\label{e:delta_cond_last2}
\end{align}

Analogously to \eqref{e:conditions_lambdamu_2} we assume the following inequalities are satisfied
\begin{equation}\label{e:conditions_lambdamu_3}
\sum_{j< q} \delta_{j,t_0} \lambda_j \leq  \delta_{q,t_0} \lambda_q, \quad
 \frac{\delta_{q,t_0}^{\sfrac{1}{2}}\lambda_q\ell}{\delta_{q+1,t_0}^{\sfrac{1}{2}}}\leq1, \quad\text{and}\quad 
\frac{\delta_{q,t_0}^{\sfrac12}\lambda_q}{\mu} \leq \lambda_{q+1}^{-\eps_1}.
\end{equation}
The last inequality being a trivial consequence of \eqref{e:conditions_lambdamu_2} and the inequality $\delta_{q,t_0}\leq \delta_q$.
Observe that we \emph{do not} assume a condition akin to the last inequality of \eqref{e:conditions_lambdamu_2}.  This remark is worth keeping in mind as we will apply  the arguments of \cite{BDIS} extensively, where such a condition was present.  Luckily, this condition is only really required at one specific point in the paper:  the estimation of 
\[\norm{\partial_t \mathring R_1+ v_1\cdot \nabla \mathring R_1 }_0,\]
for which on a subset of time we will present sharper estimates.
This condition was also used in a few isolated cases in \cite{BDIS} in order to simplify a number of terms arising from estimates, however this was primarily done for aesthetic reasons.

%%%%%%%%%%%%%%%%%%%%%%
\section{Estimates on the perturbation}\label{s:perturbation_estimates}
%%%%%%%%%%%%%%%%%%%%%%

In order to bound the perturbation, we apply nearly identical arguments used in Section 3 of \cite{BDIS}.

We recall the following notation from \cite{BDIS}
\begin{align*}
\phi_{kl}(x,t)&:= e^{i\lambda_{q+1}k\cdot[\Phi_l(x,t)-x]},\\
L_{kl}&:=a_{kl}B_k+\Bigl(\frac{i}{\lambda_{q+1}}\nabla a_{kl}-a_{kl}(D\Phi_{l}-\Id)k\Bigr)\times\frac{k\times B_k}{|k|^2}.
\end{align*}
The perturbation $w$ can then be written as
\begin{equation*}
w=\sum_{kl}\chi_l\,L_{kl}\,\phi_{kl}\,e^{i\lambda_{q+1}k\cdot x}=\sum_{kl}\chi_l\,L_{kl}\,e^{i\lambda_{q+1}k\cdot\Phi_l}\,.
\end{equation*}

 For reference we note that as a consequence of \eqref{e:v_iter}, and \eqref{e:summabilities} we have
\begin{equation}\label{e:uni_v_bound}
\norm{v_q}_0 \leq 1.
\end{equation}
We also recall that as a consequence of simple convolution inequalities together with the inequalities \eqref{e:delta_cond_first2}  we have for a fixed $t_0$, $N\geq 1$ and times $t$ satisfying  $\abs{\mu_{q+1}t-l_{q+1}}<1$
\begin{equation}
\norm{v_{\ell}}_N\leq \delta_{q,t_o}^{\sfrac12}\lambda_q\ell^{-N+1}.\label{e:v_est}
\end{equation}

With this notation we now present a minor variant of Lemma 3.1 from \cite{BDIS}.

\begin{lemma}\label{l:ugly_lemma}  Fix a time $t_0\in(-1,1)$ and let $l_{q+1}$ be as before, i.e.\ the unique integer such that $t_0\in[-\sfrac12+l_{q+1},\sfrac12+l_{q+1})$.
Assuming the series of inequalities listed in Section \ref{s:ordering} hold then we have the following estimates.  For $t$ such that  $\abs{\mu t-l_{q+1}}\leq 1$ and $l\in \{l_{q+1}-1,l_{q+1},l_{q+1}+1\}$ we have
\begin{align}
\norm{D\Phi_l}_0&\leq C\, \label{e:phi_l}\\
\norm{D\Phi_l - \Id}_0 &\leq C \frac{\delta_{q,t_0}^{\sfrac{1}{2}}\lambda_q}{\mu}\label{e:phi_l_1}\\
\norm{D\Phi_l}_N&\leq C \frac{\delta_{q,t_0}^{\sfrac{1}{2}} \lambda_q }{\mu \ell^N},& N\ge 1\label{e:Dphi_l_N}
\end{align}
Moreover,
\begin{align}
\norm{a_{kl}}_0+\norm{L_{kl}}_0&\leq C \delta_{q+1,t_0}^{\sfrac12}\label{e:L}\\
\norm{a_{kl}}_N&\leq C\delta_{q+1,t_0}^{\sfrac12}\lambda_q\ell^{1-N},&N\geq 1\label{e:Da}\\
\norm{L_{kl}}_N&\leq C\delta_{q+1,t_0}^{\sfrac12}\ell^{-N},&N\geq 1\label{e:DL}\\
\norm{\phi_{kl}}_N&\leq C \lambda_{q+1} \frac{\delta_{q,t_0}^{\sfrac{1}{2}} \lambda_q}{\mu \ell^{N-1}}
+ C \left(\frac{\delta_{q,t_0}^{\sfrac{1}{2}} \lambda_q \lambda_{q+1}}{\mu}\right)^N\nonumber \\
&\stackrel{\eqref{e:ell_def}\&\eqref{e:conditions_lambdamu_2}}{\leq} C\ell^{-N}&N\geq 1.\label{e:phi}
\end{align}
Consequently, for any $N\geq 0$
\begin{align}
\norm{w_c}_N &\leq C\delta_{q+1,t_0}^{\sfrac12} \left(\frac{\lambda_q}{\lambda_{q+1}}+\frac{\delta_{q,t_0}^{\sfrac12}\lambda_q}{\mu}\right) 
\lambda_{q+1}^N\label{e:corrector_est2}\\
&\stackrel{\eqref{e:conditions_lambdamu_2}}{\leq} C\delta_{q+1}^{\sfrac12} \frac{\delta_{q}^{\sfrac12}\lambda_q}{\mu} 
\lambda_{q+1}^N\label{e:corrector_est},\\
\norm{w_o}_N&\leq C\delta_{q+1,t_0}^{\sfrac12}\lambda_{q+1}^N\label{e:W_est_N2}\\
&\leq C\delta_{q+1}^{\sfrac12}\lambda_{q+1}^N\label{e:W_est_N}.
\end{align}
The constants appearing in the above estimates depend only on $N$ and  the constant $C_0$ given in \eqref{e:iter_rey} and \eqref{e:iter_rey2}. In particular for a fixed $N$, the constants appearing in \eqref{e:L}-\eqref{e:DL} and \eqref{e:corrector_est2}-\eqref{e:W_est_N} can be made arbitrarily small by taking $C_0$ to be sufficiently large.  Furthermore, the weaker estimates \eqref{e:corrector_est} and \eqref{e:W_est_N} hold uniformly in time.
\end{lemma}

The proof of the above lemma follows from essentially exactly the same arguments to those given in the proof of Lemma 3.1 from \cite{BDIS} -- making use of our new sequence of parameters $\delta_{q,t_0}$.  The only minor point of departure from \cite{BDIS} is the appearance of the term $\frac{\lambda_q}{\lambda_{q+1}}$ in \eqref{e:corrector_est2}.  This is related to the fact that we do not have a parameter ordering akin to the last inequality of \eqref{e:conditions_lambdamu_2} for the parameters $\delta_{q,t_0}$. Nevertheless, since $\delta_{q,t_0}\leq \delta_q$, the estimate is sharper than the corresponding estimate of \cite{BDIS} and hence we obtain \eqref{e:corrector_est}.    From the definition of $w_c$ we have 
\begin{align*}
\|w_c\|_N\leq &C\sum_{kl}\chi_l\left(\frac{1}{\lambda_{q+1}}\|a_{kl}\|_{N+1}+\|a_{kl}\|_0\|D\Phi_l-\Id\|_N+\|a_{kl}\|_N\|D\Phi_l-\Id\|_0\right)\\
&+C\|w_c\|_0\sum_l\chi_l\left(\lambda_{q+1}^N\|D\Phi_l\|_0^N+\lambda_{q+1}\|D\Phi_l\|_{N-1}\right).
\end{align*}
Hence applying \eqref{e:phi_l}-\eqref{e:Da} and applying the inequalities from Section \ref{s:ordering} we obtain \eqref{e:corrector_est2}.

\begin{corollary}\label{c:ugly_cor}
Under the assumptions of Lemma \ref{l:ugly_lemma} we have
\begin{align}
& \lambda_{q+1}^{-1}\|v_1 \|_1 + \|w \|_0 \leq  \delta_{q+1, t_0}^{\sfrac{1}{2}} \label{e:final_v_est}\\
&\lambda_{q+1}^{-2} \|p_1 \|_2 + \lambda_{q+1}^{-1}\|p_1 \|_1 + \|p_1 - p \|_0 \leq \delta_{q+1, t_0}~ .\label{e:final_p_est}
\end{align}
\end{corollary}
\begin{proof}
Recall by definition, we have the following inequalities: 
\begin{align*}
\norm{w}_N&\leq \norm{w_o}_N+\norm{w_c}_N \\
\norm{v_1}_N&\leq \norm{v}_N+\norm{w}_N\\
\norm{p_{1}-p}_N&\leq \norm{\abs{w_o}^2}_N+\norm{\abs{w_c}^2}_N
+ \norm{\ip{w_o,w_c}}_N+\norm{\abs{\ip{v-v_{\ell},w}}}_N\\
&\leq +C\norm{w_o}_N\left(\norm{w_o}_0+\norm{w_c}_0\right)+C\norm{w_c}_N\norm{w_c}_0\\
&\quad+C\norm{v-v_{\ell}}_N\norm{w}_0+C\norm{v-v_{\ell}}_0\norm{w}_N\\
\norm{p_1}_N&\leq \norm{p}_N+\norm{p_{1}-p}
\end{align*}
Hence \eqref{e:final_v_est} and \eqref{e:final_p_est} follow as a consequence of \eqref{e:delta_cond_first2},\eqref{e:delta_cond_second2}, \eqref{e:conditions_lambdamu_3}, \eqref{e:v_est}, \eqref{e:corrector_est2} and \eqref{e:W_est_N2}. 
\end{proof}

We now present a variant of Lemma 3.2 from \cite{BDIS}.  We recall from \cite{BDIS} the notation for the \emph{material derivative}: $D_t:= \partial_t + v_\ell \cdot \nabla$.
\begin{lemma}\label{l:ugly_lemma_2} Under the assumptions of Lemma \ref{l:ugly_lemma} we have
\begin{align}
\|D_t v_\ell\|_N &\leq C \delta_{q,t_0}\lambda_q(1+\lambda_q\ell^{1-N})+C\delta_{q+1,t_0}\lambda_q\ell^{-N}\, , \label{e:Dt_v2}\\
&\leq C \delta_{q,t_0}\lambda_q\ell^{-N}\label{e:Dt_v}\\
%\|D_t v\|_0 &\leq C \delta_{q,t}\lambda_q\label{e:Dt_v_bis}\, ,\\
\|D_t L_{kl}\|_N &\leq C \delta_{q+1,t_0}^{\sfrac{1}{2}} \delta_{q,t_0}^{\sfrac{1}{2}} \lambda_q\ell^{-N}\, ,\label{e:DtL}\\
\|D^2_t L_{kl}\|_N &\leq C\delta_{q+1,t_0}^{\sfrac{1}{2}} \lambda_{q}\ell^{-N}(\delta_{q,t_0} \lambda_q+\delta_{q+1,t_0}\ell^{-1})\, ,\label{e:D2tL2}\\
& \leq C\delta_{q+1,t_0}^{\sfrac{1}{2}}\delta_{q,t_0} \lambda_{q}\ell^{-N-1} \label{e:D2tL}
\end{align}
Consequently for $t$ in the range $\abs{t\mu-l_{q+1}}\le \sfrac12(1- \lambda_{q+1}^{-\eps_1})$ we have
\begin{align}
\norm{D_t w_c}_N &\leq C  \delta_{q+1,t_0}^{\sfrac12}\delta_{q,t_0}^{\sfrac12}\lambda_q\lambda_{q+1}^N\, ,\label{e:Dt_wc2}\\
\norm{D_t w_o}_N &\equiv 0\, .\label{e:Dt_wo2}
\end{align}
Moreover we have the following estimates which are valid uniformly in time
\begin{align}
\norm{D_t w_c}_N &\leq C  \delta_{q+1}^{\sfrac12}\delta_q^{\sfrac12}\lambda_q\lambda_{q+1}^{N+\eps_1}\, ,\label{e:Dt_wc}\\
\norm{D_t w_o}_N &\leq C  \delta_{q+1}^{\sfrac12}\mu\lambda_{q+1}^{N+\eps_1}\, .\label{e:Dt_wo}
\end{align}
Again, we note that the constants $C$ depend only on our choice of $C_0$: in particular, the constants appearing in \eqref{e:DtL}-\eqref{e:Dt_wo} can be made arbitrarily small by taking $C_0$ sufficiently large. 
\end{lemma}
\begin{proof}
First note that \eqref{e:DtL}, \eqref{e:Dt_wc} and \eqref{e:Dt_wo} follow by exactly the same arguments as those given in Lemma 3.2 of \cite{BDIS}  -- making use of our new sequence of parameters $\delta_{q,t_0}$.  However in contrast  \cite{BDIS}, time derivatives falling on $\chi_l$ for some $l$ pick up an additional factor of $\lambda_{q+1}^{\eps_1}$, which explains this additional factor appearing in \eqref{e:Dt_wc} and \eqref{e:Dt_wo}. 

To prove \eqref{e:Dt_v2} and \eqref{e:D2tL2}, in addition to using our new parameters $\delta_{q,t_0}$, we will take advantage of our second order inductive estimates on the pressure in order to obtain sharper estimates than those found in \cite{BDIS}.

Consider \eqref{e:Dt_v2}, we note that by the arguments of \cite{BDIS} we obtain that 
\[\|D_t v_\ell\|_N\leq \|\nabla p * \psi_\ell\|_N+\|{\rm div}\, \mathring R * \psi_\ell\|_N+C\lambda_q^2\ell^{1-N} \delta_{q,t_0}\]
Then from the inductive estimates \eqref{e:delta_cond_second2} of the pressure $p$, the estimate \eqref{e:iter_rey2} on the Reynolds stress $\mathring R$, together with standard convolution estimates, we obtain \eqref{e:Dt_v2}. From \eqref{e:ell_lambda} and since $\delta_{q+1,t_0}\leq \delta_{q,t_0}$ we obtain \eqref{e:Dt_v}.   

We now consider the estimate \eqref{e:D2tL2}.
\begin{align*}
D_t^2L_{kl}=&\Bigl(-\frac{i}{\lambda_{q+1}}(D_tDv_\ell)^T\nabla a_{kl}+\frac{i}{\lambda_{q+1}}Dv_\ell^TDv_\ell^T\nabla a_{kl}+\\
&-a_{kl}D\Phi_lDv_\ell Dv_\ell k+a_{kl}D\Phi_lD_tDv_\ell k\Bigr)\times\frac{k\times B_k}{|k|^2}.
\end{align*}
Note that $D_tDv_\ell=DD_tv_\ell-Dv_\ell Dv_\ell$, so that
\begin{align*}
\|D_tDv_\ell\|_N&\leq \|D_tv_\ell\|_{N+1}+C\|Dv_\ell\|_N\|Dv_\ell\|_0\\
&\stackrel{\eqref{e:Dt_v2}\&\eqref{e:v_est}}{\leq} C( \delta_{q,t_0}\lambda_q^2\ell^{-N}+\delta_{q+1,t_0}\lambda_q\ell^{-N-1})\left(1+\lambda_q\ell\right)\\
&\stackrel{ \eqref{e:ell_lambda}}{\leq} C \delta_{q,t_0}\lambda_q^2\ell^{-N}+C\delta_{q+1,t_0}\lambda_q\ell^{-N-1}.
\end{align*}

Hence utilizing the estimates in Lemma \ref{l:ugly_lemma} we obtain
\begin{align*}
\|D_t^2L_{kl}\|_N&\leq C\delta_{q+1,t_0}^{\sfrac{1}{2}}\lambda_q\ell^{-N}\left(\delta_{q,t_0}\lambda_q+\frac{\delta_{q+1,t_0}}{\ell}\right)\left(1+\frac{\lambda_q}{\lambda_{q+1}}+\frac{\delta_{q,t_0}^{\sfrac12}\lambda_q}{\mu}\right)\\
&\stackrel{\eqref{e:conditions_lambdamu_3}}{\leq} C\delta_{q+1,t_0}^{\sfrac{1}{2}} \lambda_q\ell^{-N}(\delta_{q,t_0} \lambda_q+\delta_{q+1,t_0}\ell^{-1}).
\end{align*}
Thus we obtain \eqref{e:D2tL2}.  
 The estimate \eqref{e:Dt_wc2} follows also as a consequence of \eqref{e:ell_lambda}, Lemma \ref{l:ugly_lemma} and \eqref{e:DtL}.
\end{proof}

\begin{remark}\label{r:pressure}
While \eqref{e:Dt_v2} and \eqref{e:D2tL2} are the analogues of the corresponding estimates in \cite{BDIS}, \eqref{e:Dt_v2} and \eqref{e:D2tL2} are sharper and are derived taking into account the bounds on the second derivatives of the pressure \eqref{e:delta_cond_second2}.
\end{remark}
%%%%%%%%%%%%%%%%%%%%%%%%%%%%%%%%%%
\section{Estimates on the Reynolds stress}\label{s:reynolds}
%%%%%%%%%%%%%%%%%%%%%%%%%%%%%%%%%%

In this section we describe the estimates on Reynolds stress, which follow by applying the arguments of Section 5 of \cite{BDIS} to the present scheme.  The main result is the following proposition, which is a sharper, time localized version of Proposition 5.1 of \cite{BDIS}, providing estimates for a subset of the times in the complement of the regions where the cut-off functions overlap.

\begin{proposition}\label{p:R} Fix $t$ in the range $\abs{t\mu-l_{q+1}}<\sfrac12(1- \lambda_{q+1}^{-\eps_1})$.  There is a constant $C$ such that, if $\delta_{q,t_0}$, $\delta_{q+1,t_0}$ and $\mu$,  satisfy  \eqref{e:conditions_lambdamu_3}, then we have
\begin{allowdisplaybreaks}
\begin{align}
\|R^0\|_0 +\frac{1}{\lambda_{q+1}}\|R^0\|_1+\frac{1}{\delta_{q+1,t_0}^{\sfrac{1}{2}}\lambda_{q+1}}\|D_tR^0\|_0&\leq C\delta_{q+1,t_0}^{\sfrac{1}{2}}\delta_{q,t_0}^{\sfrac{1}{2}} \lambda_q \ell\label{e:R0}\\
\|R^1\|_0 +\frac{1}{\lambda_{q+1}}\|R^1\|_1+\frac{1}{\delta_{q+1,t_0}^{\sfrac{1}{2}}\lambda_{q+1}}\|D_tR^1\|_0&\leq C \frac{\delta_{q+1,t_0}\delta_{q,t_0}^{\sfrac{1}{2}} \lambda_q \lambda_{q+1}^{\eps_1}}{\mu}+\nonumber\\&\qquad C\delta_{q+1,t_0}^{\sfrac{1}{2}}\delta_{q,t_0}^{\sfrac{1}{2}} \lambda_q \ell\label{e:R1}\\
\|R^2\|_0 +\frac{1}{\lambda_{q+1}}\|R^2\|_1+\frac{1}{\delta_{q+1,t_0}^{\sfrac{1}{2}}\lambda_{q+1}}\|D_tR^2\|_0&\leq C \frac{\delta_{q+1,t_0}\delta_{q,t_0}^{\sfrac{1}{2}} \lambda_q \lambda_{q+1}^{\eps_1}}{\mu} +\nonumber\\&\qquad C\delta_{q+1,t_0}^{\sfrac{1}{2}}\delta_{q,t_0}^{\sfrac{1}{2}} \lambda_q \ell\label{e:R2}\\
\|R^3\|_0 +\frac{1}{\lambda_{q+1}}\|R^3\|_1+\frac{1}{\delta_{q+1,t_0}^{\sfrac{1}{2}}\lambda_{q+1}}\|D_tR^3\|_0&\leq C \delta_{q+1,t_0}^{\sfrac{1}{2}} \delta_{q,t_0}^{\sfrac{1}{2}} \lambda_q \ell\label{e:R3}\\
\|R^4\|_0 +\frac{1}{\lambda_{q+1}}\|R^4\|_1+\frac{1}{\delta_{q+1,t_0}^{\sfrac{1}{2}}\lambda_{q+1}}\|D_tR^4\|_0&\leq C\delta_{q+1,t_0}^{\sfrac12}\delta_{q,t_0}^{\sfrac12} \lambda_q \ell\label{e:R4}\\
\|R^5\|_0 +\frac{1}{\lambda_{q+1}}\|R^5\|_1+\frac{1}{\delta_{q+1,t_0}^{\sfrac{1}{2}}\lambda_{q+1}}\|D_tR^5\|_0&\leq C \frac{\delta_{q+1,t_0} \delta_{q,t_0}^{\sfrac{1}{2}}\lambda_q}{\mu}+\nonumber\\&\qquad C\delta_{q+1,t_0}^{\sfrac{1}{2}}\delta_{q,t_0}^{\sfrac{1}{2}} \lambda_q\ell\, .\label{e:R5}
\end{align}
\end{allowdisplaybreaks}
Thus
\begin{multline}
\|\mathring{R}_1\|_0+\frac{1}{\lambda_{q+1}}\|\mathring{R}_1\|_1 + \frac{1}{\delta_{q+1,t_0}^{\sfrac{1}{2}}\lambda_q}\|D_t \mathring{R}_1\|_0\leq \\C \left(
 \frac{\delta_{q+1,t_0} \delta_{q,t_0}^{\sfrac{1}{2}} \lambda_q \lambda_{q+1}^{\eps_1}}{\mu} 
+ \delta_{q+1,t_0}^{\sfrac{1}{2}} \delta_{q,t_0}^{\sfrac{1}{2}} \lambda_q \ell\right)\, ,\label{e:allR2}
\end{multline}
\begin{multline}
\|\partial_t \mathring{R}_1 + v_1\cdot \nabla \mathring{R}_1\|_0\leq \\C \delta_{q+1,t_0}^{\sfrac{1}{2}}\lambda_{q+1} \left(\frac{\delta_{q+1,t_0} \delta_{q,t_0}^{\sfrac{1}{2}} \lambda_q \lambda_{q+1}^{\eps_1}}{\mu} 
+ \delta_{q+1,t_0}^{\sfrac{1}{2}} \delta_{q,t_0}^{\sfrac{1}{2}} \lambda_q \ell\right).\label{e:Dt_R_all2}
\end{multline}
\end{proposition}

The arguments will be a minor variation to those found in Proposition 5.1 of \cite{BDIS}, the key differences being:
\begin{itemize}
\item[(A)] Since for all $l$ we have $\chi_l'$ is identically zero for times $t$ in the range $\abs{t\mu-l_{q+1}}<\sfrac12(1- \lambda_{q+1}^{-\eps_1})$, no positive powers of $\mu$ will appear as a consequence of differentiating in time.
\item[(B)] As previously mentioned in Section \ref{s:ordering}, in contrast to the case in \cite{BDIS}  we \emph{do not} have an estimate of the type
\begin{equation}\label{e:doesnt_exist}
\frac{1}{\lambda_{q+1}}\leq \frac{\delta_{q+1,t_0}^{\sfrac12}}{\mu}
\end{equation}
at our disposal.  
\item[(C)] In many of the material derivative estimates in \cite{BDIS} the estimate $\delta_q^{\sfrac12}\lambda_q \leq \mu$ was used in order to simplify terms: we will avoid employing such an estimate, although in its place we will sometimes use the estimate $\delta_{q,t_0}^{\sfrac12}\lambda_q \leq\delta_{q+1,t_0}^{\sfrac12}\lambda_{q+1}$.
\item[(D)] In \cite{BDIS} a new constant $\epsilon>0$ was introduced in order to state the analogous estimates: in order to minimize the number of small constants, we simply use $\eps_1$ and apply the identity $\ell=\lambda_{q+1}^{\eps_1-1}$ in order to reduce the number of terms in the estimates.
\item[(E)] No term of the type 
\begin{equation}\label{e:stupid_term}
\frac{\delta_{q+1,t_0}^{\sfrac12}\delta_{q,t_0}\lambda_q}{\lambda_{q+1}^{1-\eps}\mu\ell}, 
\end{equation}
appears in the estimate \eqref{e:R0} and within the brackets of the right hand sides of \eqref{e:allR2} and \eqref{e:Dt_R_all2}.   This is related to the fact that in \cite{BDIS} the authors did not keep track of second derivatives of the pressure (see Remark \ref{r:pressure}).\footnote{ Such a term imposes strong restrictions on the choice of $\ell$ to ensure convergence and is in part the reason for the complicated choice of $\ell$ taken in \cite{BDIS}.}
\end{itemize}

\begin{proof}
Keeping in mind the observations (A), (B), (C) and (D) above, the proof of \eqref{e:R2}-\eqref{e:R5} follows by applying nearly identical arguments to that found in Proposition 5.1 of \cite{BDIS}. Indeed the estimates on $R^2$, $R^3$, $R^4$ and $R^5$, depend on the $C^0$ of $w_o$, $w_c$, $v$, $\mathring R$, $D_t w_o$, $D_t w_c$, $D_t v_{\ell}$, $D_t \mathring R$   and the $C^1$ norm of  $w_o$, $w_c$, $v$, $p$, $\mathring R$. For bounding these quantities we use the estimates \eqref{e:delta_cond_first2}-\eqref{e:delta_cond_last2}, together with the estimates from Lemmas \ref{l:ugly_lemma} and \ref{l:ugly_lemma_2}, which are analogous to the corresponding estimates ones in \cite{BDIS}.  The estimate \eqref{e:allR2} easily follows as a consequence of \eqref{e:R0}-\eqref{e:R5}, and \eqref{e:Dt_R_all2} follows from  \eqref{e:allR2} together with the observation
\[
\|\partial_t \mathring{R}_1 + v_1\cdot \nabla \mathring{R}_1\|_0 \leq \|D_t \mathring{R}_1\|_0 + \left(\|v-v_\ell\|_0 + \|w\|_0\right)\|\mathring{R}_1\|_1 \, .
\]
Therefore we will restrict ourselves to proving the estimates \eqref{e:R0} and \eqref{e:R1}.  For reasons of brevity, in what follows we adopt the abuse of notation $l_1=l_{q+1}$.

\noindent{\bf Estimates on $R^0$.}  Recall from \cite{BDIS}  that, by the definition of  $R^0$ given by \eqref{e:R^0_def}, taking into account Propositions \ref{p:stat_phase} and \ref{p:commutator} and applying the decomposition 
\begin{equation}\begin{split}
D_tR^0&=([D_t,\mathcal R]+\mathcal R D_t)(
\partial_tw+v_\ell\cdot \nabla w+w\cdot \nabla v_\ell)\\&=([v_{\ell}\cdot \nabla,\mathcal R]+\mathcal R D_t)(
\partial_tw+v_\ell\cdot \nabla w+w\cdot \nabla v_\ell)\end{split}\label{e:comm_decomp}\end{equation}
we need to bound the terms $\Omega_{kl}$ where
\begin{equation*}
\partial_tw+v_\ell\cdot \nabla w+w\cdot \nabla v_\ell=\sum_{kl}\Omega_{kl}e^{i\lambda_{q+1}k\cdot x}\, ,
\end{equation*}
that is
\[\Omega_{kl}:=\left(\chi_l'L_{kl}+\chi_lD_tL_{kl}+\chi_lL_{kl}\cdot \nabla v_\ell\right)\phi_{kl}\,.\]
and the terms $\Omega'_{kl}$ where
\begin{equation}\label{e:Omega1}
D_t \left(\partial_tw+v_\ell\cdot \nabla w+w\cdot \nabla v_\ell\right)
 :=\sum_{k}\Omega'_{kl}e^{i\lambda_{q+1}k\cdot x},
\end{equation}
that is
\begin{align}
&\Omega'_{kl}:= \Bigl(\partial_t^2\chi_lL_{kl}+2\partial_t\chi_lD_tL_{kl}+\chi_lD_t^2L_{kl}+\nonumber\\
&+\partial_t\chi_lL_{kl}\cdot\nabla v_\ell+\chi_lD_tL_{kl}\cdot\nabla v_\ell+\chi_lL_{kl}
\cdot\nabla D_tv_\ell-\chi_lL_{kl}\cdot\nabla v_\ell\cdot \nabla v_\ell\Bigr)\phi_{kl}.\label{e:Omega2} 
\end{align}
 Precisely, applying Propositions \ref{p:stat_phase} with $\alpha=\eps_1$ we obtain
\begin{align}
\| R^0\|_0 \leq& C\sum_{kl} \left(\ell \|\Omega_{kl}\|_0 + \lambda_{q+1}^{1-N}\ell \|\Omega_{kl}\|_N + 
\lambda_{q+1}^{-N} \|\Omega_{kl}\|_{N+\eps_1}\right),\label{e:R0est}\\
\| R^0\|_1 \leq& C\lambda_{q+1}\sum_{kl} \left(\ell \|\Omega_{kl}\|_0 + \lambda_{q+1}^{1-N}\ell \|\Omega_{kl}\|_N + 
\lambda_{q+1}^{-N} \|\Omega_{kl}\|_{N+\eps_1}\right)\nonumber+\\
&\sum_{k} \left(\ell \|\Omega_{kl}\|_1 + \lambda_{q+1}^{1-N}\ell \|\Omega_{kl}\|_{N+1} + 
\lambda_{q+1}^{-N} \|\Omega_{kl}\|_{N+1+\eps_1}\right)\label{e:R0Dest}
\end{align}
and  by 
  Propositions \ref{p:stat_phase} and \ref{p:commutator} and the decomposition \eqref{e:comm_decomp} we obtain
\begin{align}
\|D_t R^0\|_0 \leq&C\sum_{kl}\bigg[ \left(\ell \|\Omega_{kl}'\|_0 + \lambda_{q+1}^{1-N}\ell \|\Omega_{kl}'\|_N + 
\lambda_{q+1}^{-N} \|\Omega_{kl}'\|_{N+\eps_1}\right)\nonumber\\
&+\ell\lambda_{q+1}^{-1} \|v_{\ell}\|_1 \|\Omega_{kl}\|_1 \nonumber\\&
+  \lambda_{q+1}^{1-N}\ell \left(\|\Omega_{kl}\|_{N+\eps_1} \|v_{\ell}\|_{1+\eps_1} +\|\Omega_{kl}\|_{1+\eps_1} \|v_{\ell}\|_{N+\eps_1}\right).\label{e:DtR0est}
\end{align}

Observe that since we assumed $\abs{t\mu-l_{1}}<\sfrac12(1- \lambda_{q+1}^{-\eps_1})$ we have that $\Omega_{kl},\Omega'_{kl}\equiv 0$ for all $l\neq l_{1}$. Moreover,  since on the given temporal range $\chi_{l_{1}}\equiv 1$ and $\chi'_{l_{1}}\equiv 0$, we have
\[\Omega_{kl_{1}}:=\left(D_tL_{kl_{1}}+L_{kl_{1}}\cdot \nabla v_\ell\right) \phi_{kl_{1}}\,,\]
and
\begin{align*}
& \Omega'_{kl_{1}}:= \Bigl(D_t^2L_{kl_1}+D_tL_{kl_{1}}\cdot\nabla v_\ell+L_{kl_{1}}
\cdot\nabla D_tv_\ell-L_{kl_{1}}\cdot\nabla v_\ell\cdot \nabla v_\ell\Bigr) \phi_{kl_{1}}.
\end{align*}

Applying Lemmas \ref{l:ugly_lemma}, Lemma \ref{l:ugly_lemma_2}, \eqref{e:v_est} and \eqref{e:conditions_lambdamu_3} we obtain
\begin{align}\label{e:Omega_Est}
\|\Omega_{kl_1}\|_N\leq C \delta_{q+1,t_0}^{\sfrac12}\delta_{q,t_0}^{\sfrac12}\lambda_q\ell^{-N}
.
\end{align}

Similarly we obtain
\begin{align}\nonumber
\|\Omega'_{kl_1}\|_N&\leq C\delta_{q+1,t_0}^{\sfrac{1}{2}} \lambda_{q}\ell^{-N}(\delta_{q,t_0} \lambda_q+\delta_{q+1,t_0}\ell^{-1})\\
&\stackrel{ \eqref{e:ell_lambda}\&\eqref{e:conditions_lambdamu_3}}{\leq} C\delta_{q+1,t_0}\delta_{q,t_0}^{\sfrac12}\lambda_q\lambda_{q+1}\ell^{-N}.\label{e:extraterm}
\end{align}
\begin{remark}\label{r:pressure2}
Note we implicitly used the estimates \eqref{e:Dt_v2} and \eqref{e:D2tL2} which rely on second order inductive estimates on the pressure (see Remark \ref{r:pressure}). In \cite{BDIS}, only first order estimates of the pressure were assumed, resulting in an additional error term of the type \eqref{e:stupid_term}.
\end{remark}

Hence choosing $N$ large enough such that $N\eps_1> 3$, then combining \eqref{e:R0est}-\eqref{e:extraterm} we obtain \eqref{e:R0}.

\medskip

\noindent{\bf Estimates on $R^1$.}  Recall that a key ingredient to the estimation of $R^1$ involves estimating
\[f_{klk'l'}:=\chi_l\chi_{l'}a_{kl}a_{k'l'}\phi_{kl}\phi_{k'l'}\]
and
\[D_t\left(\nabla f_{klk'l'}\,e^{i\lambda_{q+1}(k+k')\cdot x}\right)e^{-i\lambda_{q+1}(k+k')\cdot x}= \Omega''_{klk'l'}.\]

 More precisely, using Proposition \ref{p:stat_phase}, it was shown in \cite{BDIS} that
\begin{align}
\|R^1\|_0 \leq& C\mathop{\sum_{(k,l),(k',l')}}_{k+k'\neq 0} \left(\ell \|f_{klk'l'}\|_1 + \lambda_{q+1}^{1-N}\ell \|f_{klk'l'}\|_{N+1} + 
\lambda_{q+1}^{-N} [f_{klk'l'}]_{N+1+\eps_1}\right)\nonumber
\end{align}
\begin{align}
&\|R^1\|_1 \leq\nonumber \\&\quad C\lambda_{q+1}\mathop{\sum_{(k,l),(k',l')}}_{k+k'\neq 0} \left(\ell \|f_{klk'l'}\|_1 + \lambda_{q+1}^{1-N}\ell \|f_{klk'l'}\|_{N+1} + 
\lambda_{q+1}^{-N} [f_{klk'l'}]_{N+1+\eps_1}\right)\nonumber\\
&\quad+C\mathop{\sum_{(k,l),(k',l')}}_{k+k'\neq 0} \left(\ell \|f_{klk'l'}\|_2 + \lambda_{q+1}^{1-N}\ell \|f_{klk'l'}\|_{N+2} + 
\lambda_{q+1}^{-N} [f_{klk'l'}]_{N+2+\eps_1}\right)\nonumber
\end{align}
and using Proposition \ref{p:stat_phase} and \eqref{p:commutator} together with the identity $D_t\mathcal R= [v_{\ell}\cdot \nabla, \mathcal R]+\mathcal RD_t$ that
\begin{align}
\|D_t R^1\|_0 &\leq C\mathop{\sum_{(k,l),(k',l')}}_{k+k'\neq 0} \left(\ell \|\Omega''_{klk'l'}\|_1 + \lambda_{q+1}^{1-N}\ell \|\Omega''_{klk'l'}\|_{N+1} + 
\lambda_{q+1}^{-N} [\Omega''_{klk'l'}]_{N+1+\eps_1}\right)\nonumber\\
&+\ell\lambda_{q+1}^{-1} \|f_{klk'l'}\|_2 \|v_{\ell}\|_1 \nonumber\\&
+  \lambda_{q+1}^{1-N}\ell \left(\|f_{klk'l'}\|_{N+1+\eps_1} \|v_{\ell}\|_{1+\eps_1} +\|f_{klk'l'}\|_{2+\eps_1} \|v_{\ell}\|_{N+\eps_1}\right)\nonumber.
\end{align}

Again  as a consequence of our assumption $\abs{t\mu-l_{1}}<\sfrac12(1- \lambda_{q+1}^{-\eps_1})$ we have that if either $l\neq l_1$ or $l'\neq l_1$ then $f_{klk'l'}\equiv 0$ and $\Omega''_{klk'l'}\equiv 0$.  Moreover we have
\begin{align*}
\Omega_{kl_1k'l_1}'':=&-\left(a_{kl_1}Dv_\ell^T\nabla a_{k'l_1}+a_{k'l_1}Dv_\ell^T\nabla a_{kl_1}\right)\phi_{kl_1}\phi_{k'l_1}\\& -\lambda_{q+1}a_{kl_1}a_{k'l_1}\left(D\Phi_lDv_\ell^Tk+D\Phi_{l_1}Dv_\ell^Tk'\right) \phi_{kl_1}\phi_{k'l_1}.
\end{align*}

Estimating $f_{kl_1k'l_1}$ and $\Omega''_{kl_1k'l_1}$ we have from Lemma \ref{l:ugly_lemma} and Lemma \ref{l:ugly_lemma_2} for $N\geq 1$
\begin{equation}
\norm{f_{kl_1k'l_1}}_N\leq C\delta_{q+1,t_0}\ell^{1-N}\left(\lambda_q+\frac{\delta_{q,t_0}^{\sfrac12}\lambda_q \lambda_{q+1}}{\mu}\right),
\end{equation}
and  
\begin{align}
\norm{\Omega''_{kl_1k'l_1}}_0&\leq C\delta_{q+1,t_0}\delta_{q,t_0}^{\sfrac12}\lambda_q\left(\lambda_q+ \lambda_{q+1}\right)\leq C\delta_{q+1,t_0}\delta_{q,t_0}^{\sfrac12}\lambda_q\lambda_{q+1}\\
\norm{\Omega''_{kl_1k'l_1}}_N&\leq C\delta_{q+1,t_0}\delta_{q,t_0}^{\sfrac12}\lambda_q\lambda_{q+1}\ell^N,
\end{align}
for $N\geq 1$.

Combining the above estimates and again selecting $N$ such that $N\eps_1> 3$ we obtain \eqref{e:R1}.
\end{proof}

We now state uniform estimates for the new Reynolds stress. Taking advantage of some of the additional observations used previously to prove Proposition \ref{p:R},  by applying nearly identical arguments to that of Proposition 5.1 of \cite{BDIS} we obtain the following Proposition.
\begin{proposition}There is a constant $C$ such that, if $\delta_{q}$, $\delta_{q+1}$ and $\mu$,  satisfy  \eqref{e:conditions_lambdamu_2}, then we have
\begin{equation}
\|\mathring{R}_1\|_0+\frac{1}{\lambda_{q+1}}\|\mathring{R}_1\|_1 + \frac{1}{\mu}\|D_t \mathring{R}_1\|_0\leq C \left( \delta_{q+1}^{\sfrac{1}{2}} \mu\ell\lambda_{q+1}^{\eps_1} 
+ \frac{\delta_{q+1} \delta_q^{\sfrac{1}{2}} \lambda_q \lambda_{q+1}^{\eps_1}}{\mu} 
\right)\, ,\label{e:allR}
\end{equation}
\begin{equation}
\|\partial_t \mathring{R}_1 + v_1\cdot \nabla \mathring{R}_1\|_0\leq C \delta_{q+1}^{\sfrac{1}{2}}\lambda_{q+1} \left( \delta_{q+1}^{\sfrac{1}{2}} \mu\ell\lambda_{q+1}^{2\eps_1} 
+ \frac{\delta_{q+1} \delta_q^{\sfrac{1}{2}} \lambda_q \lambda_{q+1}^{2\eps_1}}{\mu} 
\right).\label{e:Dt_R_all}
\end{equation}
\end{proposition}
 
In contrast to \cite{BDIS}, the extra factors of $\lambda_{q+1}^{\eps_1}$ appearing in \eqref{e:allR} and \eqref{e:Dt_R_all} are due to the fact that in the present scheme derivatives falling on $\chi_l$ pick up an extra factor of $\lambda_{q+1}^{\eps_1}$. A second point of difference to \cite{BDIS} is that unlike \cite{BDIS}, no terms of the form
\begin{equation}\label{e:absent_terms}
\delta_{q+1}^{\sfrac{1}{2}} \delta_q^{\sfrac{1}{2}} \lambda_q \ell+\frac{\delta_{q+1}^{\sfrac12}\delta_q\lambda_q}{\lambda_{q+1}^{1-\eps}\mu}
\end{equation}
appear within the brackets of the right hand sides of \eqref{e:allR} and \eqref{e:Dt_R_all}. The absence of the first term in \eqref{e:absent_terms} can be easily explained by the fact that by \eqref{e:conditions_lambdamu_2} we have
\[\delta_{q+1}^{\sfrac{1}{2}} \delta_q^{\sfrac{1}{2}} \lambda_q \ell\leq \frac{\delta_{q+1} \delta_q^{\sfrac{1}{2}} \lambda_q \lambda_{q+1}^{\eps_1}}{\mu}. \]
The absence of the second term in  \eqref{e:absent_terms} follows by the same reasoning as the absence of an analogous term in Proposition \ref{p:R} (see the comment (D) after the statement of Proposition \ref{p:R} and Remarks \ref{r:pressure} and \ref{r:pressure2}).

\section{Choice of the parameters and conclusion of the proof}\label{s:conclusion}
We begin by noting that we have not imposed any upper bounds on the choice of $\lambda_0$ and thus we are free to choose $\lambda_0$ to be as large as need be: in what follows we will use this fact multiple times without further comment.

\subsection*{$1/5-\eps$ convergence}  We now make the following parameter choices
\begin{align*}
\alpha&:=1+\eps_0, & \lambda_q&= \floor{\lambda_0^{\alpha^q}},\\ 
\eps_1&:=\frac{\eps_0^2}{18},&
\delta_q &:= \lambda_q^{-\sfrac25+2\eps_0},\\
 \mu&:= \delta_q^{\sfrac{1}{4}}\delta_{q+1}^{\sfrac{1}{4}} \lambda_q^{\sfrac{1}{2}} \lambda_{q+1}^{\sfrac{1}{2}},
\end{align*}
here $\floor{a}$ denotes the largest integer smaller than $a$. It is worth noting that with the above choices, our definition of $\mu$ agrees with the definition given in \cite{BDIS}.

Having made the above choices it is clear the inequalities \eqref{e:mu_lower} and \eqref{e:conditions_lambdamu_2} are satisfied. Moreover assuming \eqref{e:delta_cond_first}-\eqref{e:delta_cond_last}, it follows as a consequence of 
 Corollary \ref{c:ugly_cor} that \eqref{e:delta_cond_first} and \eqref{e:delta_cond_second} are satisfied with $q$ replaced by $q+1$.  In order to show \eqref{e:iter_rey} and \eqref{e:delta_cond_last} with $q$ replaced with $q+1$ we note that with our choices of parameters we obtain from \eqref{e:allR} and \eqref{e:Dt_R_all} that
\begin{align*}
\norm{\mathring R_{q+1}}_0+\frac{1}{\lambda_{q+1}}\norm{\mathring R_{q+1}}_1&\leq C\delta_{q+1}^{\sfrac12}\mu\ell\lambda_{q+1}^{\eps_1}\\
\frac1{\delta_{q+1}^{\sfrac12}\lambda_{q+1}}\norm{(\partial_t+v_{q+1}\cdot \nabla)\mathring R_{q+1}}_0&\leq C\delta_{q+1}^{\sfrac12}\mu\ell\lambda_{q+1}^{2\eps_1}.
\end{align*}
Hence since 
\[\delta_{q+1}^{\sfrac12}\mu\ell\lambda_{q+1}^{2\eps_1}\leq C\lambda_q^{-\sfrac25+\frac{6\eps_0}{5}+\frac{5\eps_0^2}{3}+\frac{\eps_0^3}{6}},\\
\leq C \delta_{q+2}\lambda_q^{-\eps_0^2}\]
we obtain both \eqref{e:iter_rey}, and \eqref{e:delta_cond_last} with $q$ replaced by $q+1$.
Since the inequalities \eqref{e:delta_cond_first}-\eqref{e:delta_cond_last} hold for $q=0$, we obtain by induction that the inequalities hold for $q \in \N$.   The inequalities \eqref{e:v_iter}-\eqref{e:R_iter} with $q$ replaced by $q+1$ then follow as a consequence of Corollary \ref{c:ugly_cor}, together with the estimates on $v$, $p$, $\mathring R$, $w$, $w_o$ and $w_c$.
In particular, one may derive time derivative estimates on $w$ and $p_1-p$ from the simple decomposition $\partial_t=D_t-v_{\ell}\cdot\nabla$ and the estimates
\begin{align*}
\norm{\partial_t w}_0&\leq \norm{\partial_t w_o}_0+\norm{\partial_t w_c}_0\\
&\stackrel{\eqref{e:uni_v_bound}}{\leq}  \norm{D_t w_o}_0+\norm{D_t w_c}_0+\norm{w_o}_1+\norm{w_c}_1
\end{align*}
and
\begin{align*}
\|\partial_t (p_{q+1} - p_q)\|_0 &\leq (\|w_c \|_0 + \|w_o\|_0) (\|\partial_t w_c\|_0+ \|\partial_t w_o\|_0)\\
&\qquad + 2 \|w\|_0\|\partial_t v\|_0 
+ \ell\|v\|_1\|\partial_t w\|_0\\
&\stackrel{\eqref{e:uni_v_bound}}{\leq} (\|w_c \|_0 + \|w_o\|_0) (\norm{D_t w_o}_0+\norm{D_t w_c}_0+\norm{w_o}_1+\norm{w_c}_1)\\
&\qquad + 2 \|w\|_0(\|\partial_t v+v\cdot\nabla v\|_0+\|v\|_1) 
+ \ell\|v\|_1\|\partial_t w\|_0\\
&\leq (\|w_c \|_0 + \|w_o\|_0) (\norm{D_t w_o}_0+\norm{D_t w_c}_0+\norm{w_o}_1+\norm{w_c}_1)\\
&\qquad + 2 \|w\|_0(\|p\|_1+\|\mathring R\|_1+\|v\|_1)
+ \ell\|v\|_1\|\partial_t w\|_0
\end{align*}
The required estimates then follow as a consequence from \eqref{e:delta_cond_first}-\eqref{e:iter_rey}, \eqref{e:conditions_lambdamu_2}, \eqref{e:corrector_est}, \eqref{e:W_est_N}, \eqref{e:Dt_wc} and \eqref{e:Dt_wo}.

\subsection*{$1/3-\eps$ convergence}  Let us define 
 $U^{(q)}$ to be the set
\[U^{(q)}=\bigcup_{l\in [-\mu_q,\mu_q]}[\mu_q^{-1}(l+\sfrac12-\lambda_q^{-\eps_1}),\mu_q^{-1}(l+\sfrac12+\lambda_q^{-\eps_1})],\]
i.e.\ a union of $\sim2\mu_q$ balls of radius $\lambda_{q}^{-\eps_1}\mu_q^{-1}$
 and define
\[V^{(q)}=\bigcup_{q'=q}^\infty U^{(q')}. \]
Observe that $V^{(q)}$ can be covered by a sequence of balls of radius $r_i$ such that
\begin{equation}\label{e:Hausdorff_est}
\sum r_i^d\leq 3\sum_{q'=q}^{\infty} \lambda_{q'}^{-d\eps_1}\mu_{q'}^{1-d}.
\end{equation}
Thus assuming
\begin{equation}\label{e:d_ineq}
d>\frac{(1+\alpha)(-\frac15+\eps_0+1)}{(1+\alpha)(-\frac15+\eps_0+1)+2\alpha\eps_1},
\end{equation}
it follows that the right hand side of \eqref{e:Hausdorff_est} converges to zero as $q$ tends to infinity.  

From this point on we assume $d<1$ is fixed, satisfying \eqref{e:d_ineq} -- which we note is possible due to the fact the right hand side of \eqref{e:d_ineq} is strictly less than $1$. 

For any time $t_0\in \bigcap_N V^{(N)}$  we simply set  $\delta_{q,t_0}=\delta_q$ for all $q$.

Now suppose $t_0\notin V^{(N)}$ for some integer $N$, furthermore assume $N$ to be the smallest such integer.  We now make the following parameter choices
\begin{equation*}
  \delta_{q+1,t_0}:=\begin{cases}
   \lambda_{q+1}^{-\sfrac25+2\eps_0} & \text{if } q \leq N \\
   \max\left(\lambda_q^{-\frac{\eps_0^2} 9} \delta_{q,t_0}^\alpha, \lambda_{q+1}^{-\sfrac23+2\eps_0}\right),       & \text{if } q > N
  \end{cases}
\end{equation*}
It follows that
\[ \frac{\delta_{q,t_0}\lambda_q}{\delta_{q+1,t_0}\lambda_{q+1}}\geq \lambda_q^{\sfrac{\eps_0}3},\]
from which we obtain \eqref{e:conditions_lambdamu_3} assuming $\eps_0$ is sufficiently small.  Applying   Corollary \ref{l:ugly_lemma_2} and Proposition \ref{p:R} iteratively we see that (\ref{e:delta_cond_first2}-\ref{e:delta_cond_last2}) hold for all $q\ge N$.  In particular, in order to show \eqref{e:iter_rey2} for $q$ replaced by $q+1$ we note  that by  Proposition \ref{p:R} we have for all times $t$ satisfying $\abs{t\mu_{q+1}-l_{q+1}}<\sfrac12(1- \lambda_{q+1}^{-\eps_1})$
\begin{equation}\label{e:wanted}
\|\mathring{R}_1\|_0+\frac{1}{\lambda_{q+1}}\|\mathring{R}_1\|_1 \leq \\\underbrace{C \delta_{q+1,t_0}^{\sfrac{1}{2}}\delta_{q,t_0}^{\sfrac{1}{2}}\lambda_q\ell}_I 
+ \underbrace{C\frac{\delta_{q+1,t_0} \delta_{q,t_0}^{\sfrac{1}{2}} \lambda_q \lambda_{q+1}^{\eps_1}}{\mu}}_{II}.
\end{equation}
Notice that if $\abs{\mu_{q+2}t-l_{q+2}} \leq 1$ then
\begin{align*}
\abs{t\mu_{q+1}-l_{q+1}}&\leq \frac{\mu_{q+1}}{\mu_{q+2}}\abs{\mu_{q+2}t-l_{q+2}}+\abs{\frac{\mu_{q+1}l_{q+2}}{\mu_{q+2}}-l_{q+1}}  \\
&\leq \frac{\mu_{q+1}}{\mu_{q+2}}+\mu_{q+1}\abs{\frac{l_{q+2}}{\mu_{q+2}}-t_0}+\abs{\mu_{q+1}t_0-l_{q+1}}\\
&< \frac{2\mu_{q+1}}{\mu_{q+2}}+\abs{\mu_{q+1}t_0-l_{q+1}}\\
&< 2\lambda_{q}^{-\sfrac{\eps_0}4}+\sfrac12-\lambda_{q+1}^{-\eps_1}\\
&<\sfrac12(1- \lambda_{q+1}^{-\eps_1}).
\end{align*}
Thus \eqref{e:wanted} holds for times $t$ in the range $\abs{\mu_{q+2}t-l_{q+2}}<1$. 

Taking logarithms of $I$ and $II$ we obtain
\begin{equation}\label{e:log1}
\ln I\leq \left(1+\frac{\eps_0}{2}\right)\ln \delta_{q,t_0}+\left(\frac{\eps_0^2}{ 18}+\frac{\eps_0^3}{ 18}-\eps_0\right)\ln\lambda_q+C
\end{equation}
and
\begin{equation}\label{e:log2}
\ln II \leq \left(\frac32+\eps_0\right) \ln \delta_{q,t_0}+\left(\frac15-\frac{7\eps_0}{5}-\frac{4\eps_0^2}{9}+O(\eps_0^3)\right)\ln \lambda_q+ C.
\end{equation}
Note by definition we have
\begin{equation}\label{e:log3}
\ln \delta_{q+2,t} \geq  \left(1+\eps_0\right)^2\ln \delta_{q,t_0}-\left(\frac{2\eps_0^2}{ 9}+O(\eps_0^3)\right)\ln \lambda_{q}.
\end{equation}
Thus since $\delta_{q,t_0}\geq \lambda_q^{-\sfrac23+2\eps_0}$, combining \eqref{e:log1} and \eqref{e:log3} we obtain
\begin{align}
\ln\left(\frac{I}{\delta_{q+2,t}}\right)&\leq \left(-\frac{3\eps_0}{2}-\eps_0^2\right)\ln \delta_{q,t_0}+\left(\frac{5\eps_0^2}{18}-\eps_0+O(\eps_0^3)\right)\ln\lambda_q+C\nonumber\\
&\leq \left(-\frac{\eps_0^2}{4}+O(\eps_0^3)\right)\ln\lambda_q+C.\label{e:log4}
\end{align}
Similarly, since $\delta_{q,t_0}\leq \lambda_q^{-\sfrac25+2\eps_0}$, combining \eqref{e:log2} and \eqref{e:log3} we obtain
\begin{align}
\ln\left(\frac{II}{\delta_{q+2,t}}\right)&\leq  \left(\frac12-\eps_0-\eps_0^2\right) \ln \delta_{q,t_0}+\left(\frac15-\frac{7\eps_0}{5}-\frac{2\eps_0^2}{9}+O(\eps_0^3)\right)\ln \lambda_q+ C\nonumber\\
&\leq \left(-\eps_0^2+O(\eps_0^3)\right)\ln \lambda_q+ C.\label{e:log5}
\end{align}
Hence assuming $\eps_0$ is sufficiently small, from \eqref{e:log4} and \eqref{e:log5} we obtain \eqref{e:iter_rey2} for $q$ replaced by $q+1$.

Observe also that there exists an $N'$ such that for all $q\geq N+N'$ we have
\[\delta_{q,t_0}=\lambda_q^{-\sfrac23+2\eps_0},\]
and hence the inequality \eqref{e:onsager_est} is never satisfied for $q\geq N+ N'$.  Thus
\[\Xi^{N+N'}\subset V^{N}.\]
 In particular $N'$ can be chosen universally, independent of $N$. 
Fixing $\delta>0$ and choosing $N$ such that $V^N$ can be covered by a sequence of balls of radius $r_i$ satisfying 
\[\sum r_i^d<\delta,\]
we obtain that if we set $M=N+N'$ then \eqref{e:real_Haus} is satisfied which concludes the proof of Proposition \ref{p:iterate}.

\begin{remark}
For the sake of completeness we note that analogously to the estimates \eqref{e:v_iter}-\eqref{e:R_iter}, the estimates \eqref{e:sharp1}-\eqref{e:sharp3} follow as a consequence of Lemma \ref{l:ugly_lemma}, Lemma \ref{l:ugly_lemma_2} and Proposition \ref{p:R} -- here the set $\Omega$ can be taken explicitly to be
\[\Omega:= \bigcap_{q=1}^\infty V^{(q)}.\]
\end{remark}
\bibliographystyle{acm}

\end{document}